\documentclass[a4paper,11pt]{amsart}

\usepackage{amsmath}
\usepackage{amsthm}
\usepackage{amsfonts}
\usepackage{amssymb}

\usepackage{tikz}
\usetikzlibrary{arrows}
\usepackage{enumerate}
\usetikzlibrary{calc}
\usepackage{colonequals}
\usepackage[all]{xy}
\usepackage{verbatim}

\newcommand{\inv}{^{-1}}

\newcommand{\N}{\mathbb N}

\newcommand{\R}{\mathbb R}
\newcommand{\bS}{\mathbb S}

\newcommand{\cone}{\operatorname{cone}}
\newcommand{\diam}{\operatorname{diam}}
\newcommand{\card}{\operatorname{card}}

\newcommand{\abs}[1]{\lvert#1\rvert}
\newcommand{\norm}[1]{\lVert#1\rVert}

\newtheorem{theorem}{Theorem}[section]
\newtheorem*{theorem*}{Theorem}{\bf}{\it}
\newtheorem{proposition}[theorem]{Proposition}
\newtheorem*{proposition*}{Proposition}{\bf}{\it}
\newtheorem{lemma}[theorem]{Lemma}
\newtheorem*{lemma*}{Lemma}{\bf}{\it}

\theoremstyle{definition}
\newtheorem{definition}[theorem]{Definition}
\newtheorem*{definition*}{Definition}
\theoremstyle{remark}
\newtheorem{remark}[theorem]{Remark}
\numberwithin{equation}{section}
\newtheorem{example}[theorem]{Example}
\numberwithin{equation}{section}

\begin{document}

\title[Proper branched coverings]{On proper branched coverings and a question of Vuorinen}

\author{Aapo Kauranen}
\address{ 
  Department de Matem\`atiques, Universitat Aut\`onoma de Barcelona, 08193, 
  Bellaterra (Barcelona), Spain}
\email{aapo.p.kauranen@gmail.com}
\thanks{A.K.
  acknowledges financial support from the Spanish Ministry of Economy and 
  Competitiveness, through the ``Marí\'{\i}a de Maeztu'' Programme for Units of 
  Excellence in R\&D (MDM-2014-0445) and MTM-2016-77635-P (MICINN, Spain).}

\author{Rami Luisto}
\address{Department of Mathematics and Statistics, P.O. Box 35, FI-40014 University of Jyv\"askyl\"a, Finland}
\email{rami.luisto@gmail.com}
\thanks{R.L.\
  was partially supported by the Academy of Finland
  (grant 288501 `\emph{Geometry of subRiemannian groups}')
  and by the European Research Council
  (ERC Starting Grant 713998 GeoMeG `\emph{Geometry of Metric Groups}').
}

\author{Ville Tengvall}
\address{Department of Mathematics and Statistics, P.O. Box 68 (Pietari Kalmin katu 5), FI-00014 University of Helsinki, Finland}
\email{ville.tengvall@helsinki.fi}
\thanks{The research of V.T.\ was supported by the Academy of Finland, project number 308759.}

\date{\today}

\begin{abstract}
  We study global injectivity of proper branched coverings defined on the Euclidean $n$-ball in the case when
  the branch set is compact. In particular we show that such mappings are homeomorphisms when $n=3$ or
  when the branch set is empty. This proves the corresponding cases of a question of Vuorinen
  from \cite{Vuorinen2}.
\end{abstract}

\maketitle

\section{Introduction}
\label{sec:Intro}

For a continuous, open and discrete mapping
$f \colon \Omega \to \Omega'$ between Euclidean domains
we define its \emph{branch set}, denoted $B_f$, to be the set of points where $f$ is not locally injective.
The structure of this set is connected to the topology and geometry of the mapping itself,
but in general the structure of the branch set is not well understood.
Even for the important special class of continuous, open and discrete maps called quasiregular mappings
many properties of the branch set remain largely unknown,
but the topic garners great interest. In Heinonen's
ICM address, \cite[Section 3]{HeinonenICM}, he asked the following: 
\begin{quote}
  Can we describe the geometry and the topology of the allowable branch sets of quasiregular
  mappings between metric $n$-manifolds?
\end{quote}

In this paper we focus on a particular aspect of this general question of Heinonen known as \emph{Vuorinen's question} concerning the compactness of the branch set of \emph{proper} continuous, open and discrete mappings.
The question is as follows:
\begin{quote}
  Suppose that $f \colon B^n \to f(B^n) \subset \R^n$, $n \geq 3$, is a proper continuous,
  open and discrete mapping with a compact branch set $B_f$. Is $f$ then a homeomorphism?
\end{quote}
In this paper we will refer to this question simply as \emph{the Vuorinen question}.
The question first appeared in the work of Vuorinen \cite[Remarks 3.7]{Vuorinen2} on the boundary behavior
of quasiregular mappings. Later it was also stated in the well-known monograph \cite[p.193, (4)]{Vuorinen}
and the query \cite{VuorinenQueries} of the same author. It was further promoted by Srebro in a survey collection
\cite[p. 108]{QuasiconformalSpaceMappings}, and also given in the slightly stronger setting of quasiregular mappings
in the collection \cite[p. 503, 7.66]{BarthBrannanHayman} of research problems in complex analysis. Our first main result gives a positive answer to the question in dimension three.
\begin{theorem}\label{thm:VuorinenDim3}
  Let $f \colon B^3 \to f(B^3) \subset \R^3$ be a proper continuous, open and discrete map. If $B_f$ is compact,
  then $f$ is a homeomorphism.
\end{theorem}
The crucial idea of the proof is to investigate the existence of the torsion elements of the
fundamental group of the target of the underlying map. If such elements do not exists then the
mapping is a homeomorphism; for the precise statement see Proposition \ref{prop:NoTorsionImpliesVuorinen}.
We furthermore show that the claim is true in all dimensions
when the branch set is empty.
\begin{theorem}\label{thm:VuorinenNoBranch}
  Let $f \colon B^n \to f(B^n) \subset \R^n$ be a proper continuous, open and discrete map with $n \ge 2$. If $B_f=\emptyset$,
  then $f$ is a homeomorphism.
\end{theorem}
These results should be contrasted with Zorich's global homeomorphism theorem, see e.g.\ \cite{Zorich} or
\cite[Corollary III.3.8]{Rickman-book}, which states that if $n\geq 3$, then
an entire quasiregular mapping $\R^n \to \R^n$ with an empty branch set is always a quasiconformal mapping, i.e.\ a globally
injective quasiregular mapping. In addition, the origin of Vuorinen's question is in the study of induced boundary mappings
of closed quasiregular mappings,  see \cite[section~5]{VuorinenThesis} and especially \cite[Theorem 5.3]{VuorinenThesis}. In
this context our results can be used to study whether it is possible to produce extra branching to a quasiregular mapping by
changing the mapping only locally. This would be especially desirable when we are constructing new quasiregularly elliptic
manifolds as connected sums of other quasiregularly elliptic manifolds. However, at least in three dimensions
Theorem~\ref{thm:VuorinenDim3} prohibits the non-global topological modifications of these mappings by providing a positive
answer to the following question from \cite[Open problem 9.18, p. 125]{Vuorinen}:
\begin{quote}
  Let $f \colon \Omega \to \R^n$ be a branched covering with $\Omega \subset \R^n$ a domain. Suppose
  $x_0 \in \Omega$ and $r \in (0, d(x_0,\partial \Omega))$. If $B(x_0,r)$ is a normal domain of
  $f$ and $B_f \cap \partial B(x_0,r) = \emptyset$, is $f|_{B(x_0,r)}$ then necessarily injective?
\end{quote}

In our proof of Theorem~\ref{thm:VuorinenDim3}, a crucial step is to show that
in all dimensions the claim follows whenever the target has torsion-free fundamental group at infinity,
see Definition \ref{def:TorsionFreeAtInfinity} and Proposition \ref{prop:NoTorsionImpliesVuorinen}.
The Vuorinen question in three dimensions then follows by noting that any domain in $\R^3$ has
a torsion-free fundamental group. 
In contrast, in higher dimensions the fundamental group of a Euclidean domain can have torsion elements,
this is exemplified e.g. with a tubular neighborhood of a projective plane embedded in $\R^4$,
and so our proof does not generalize to all dimensions. Another Euclidean domain with
torsion non-free fundamental group is studied in Example \ref{ex:Nemesis}.
Thus the Vuorinen question is still open in dimensions four and above in the general case.
For the case when branch set is empty, Theorem~\ref{thm:VuorinenNoBranch},
the proof also relies in the study of the existence of torsion
elements in the fundamental group of the target and we rely on the theory of $K(G,1)$-spaces; see Section \ref{sec:BranchlessVuorinen}.

We note that the proof of Theorem \ref{thm:VuorinenDim3} actually gives
rise to the following result.
\begin{proposition}
  Let $f \colon M \to f(M) \subset \R^3$ be a proper branched covering where
  $M$ is an open $3$-manifold simply connected at infinity. Suppose
  $B_f$ is compact. Then $f$ is a homeomorphism.
\end{proposition}
Likewise we note that our second main result, Theorem \ref{thm:VuorinenNoBranch},
is a corollary of the following proposition, which can be proved identically.
\begin{proposition}
  Let $f \colon M \to f(M) \subset \R^n$ be a proper branched covering where
  $M$ is an open $n$-manifold. If $M \setminus f \inv (f (B_f))$ has a contractible
  universal cover, then $f$ is a homeomorphism.
\end{proposition}

Note that the question of Vuorinen is not true in dimension two unless the branch set is assumed
to be empty as is demonstrated by the planar winding map. However, in higher dimensions the winding
map does not serve as a counterexample as the branch set of this map is an $(n-2)$-dimensional
plane and thus not compact. Nonempty compact branch sets are also possible to construct for continuous, open and
discrete mappings $B^n \to \R^n$, see e.g. \cite{KauranenLuistoTengvall}, but the known examples are no
longer proper maps. In addition, noninjective local homeomorphisms $B^n \to \R^n$ can be exemplified with the mapping
\begin{align*}
  \psi \colon (0,4\pi)^2 \times (0,1) \to \R^3,
  \quad
  \psi(z,t)\mapsto (\exp(z),t)
\end{align*} 
and its higher-dimensional analogues, but this mapping fails to be proper as well.
Note that all the above mentioned mappings are quasiregular mappings as well.

The examples in the preceding paragraph seem to hint that the challenge in the solution,
or a possible counterexample, to the Vuorinen question might lie in trying to
balance the  properness of the mapping with the compactness of the branch set.
Furthermore we note that, as Proposition \ref{prop:NoTorsionImpliesVuorinen} demonstrates,
a possible counterexample to the Vuorinen question must display some nontrivial structure
as the target of the map must have a complicated boundary, in some sense.
We also remark that the branch set itself can also exhibit very complicated structure; indeed,
on the one hand, for a continuous, open and discrete mappings whose branch set image is topologically
piecewise linear, the mapping is itself locally equivalent to a combination of
winding maps; see e.g.\ \cite{LuistoPrywes}.
On the other hand, there are mappings which do not
exhibit such simple behavior, notably the Heinonen-Rickman
map whose branch set contains a wild Cantor set (\cite{HeinonenRickman2}) and the
classical example of Church and Timourian
from \cite{ChurchTimourian} which is based on deep work of Cannon and Edwards, see
e.g.\ \cite{Cannon} and the references within.
We discuss more about the example of Church and Timourian in Example \ref{ex:Nemesis}.

\section*{\textbf{Acknowledgments}}

All of the authors would like to express their gratitude to Pekka Pankka for his valuable insights, comments and
encouragement throughout this project. Furthermore, communications with Matti Vuorinen,
during which he has shared his expertise concerning the problem, have been invaluable for our work.
We also extend our thanks to Martina Aaltonen, Jos\'e Andr\'es Rodriguez Migueles, Mike Miller and Kai Rajala for discussion
and advice related to the topic.

\section{Preliminaries}
\label{sec:Preli}

\subsection{Notation}

Throughout this paper we denote by $\Omega \subset \R^n$ a domain in $n$-dimensional Euclidean space $\R^n$ with $n \ge 2$. A point $x \in \R^n$ in coordinates is denoted as $x = (x_1, x_2, \ldots, x_n)$, and its Euclidean norm is denoted by $\abs{x} \colonequals  \sqrt{\sum_{i=1}^n x_i^2}$. An $n$-dimensional (open) ball in $\R^n$ of radius $r> 0$, centered at $a \in \R^n$ is denoted by
$$B^n(a,r) \colonequals  \{ a \in \R^n : \abs{x-y} < r  \},$$
and if the ball is centered at the origin we sometimes denote it by $B_r^n$ or by $B^n$ when $r=1$. If the dimension of the ball does not play a role we may exclude it from the notation. Moreover, if we want to emphasize that a ball $B(a,r)$ needs to be considered as a ball of some metric space $Y$ we may denote $B_Y(a,r)$. The topological interior of a set $A \subset \R^n$ will be denoted $A^\circ$ and the closure by $\overline{A}$. The topological boundary of a set $A$ is denoted by $\partial A$. The number of points in a set $A \subset \R^n$ is denoted by
$$\#(A) \colonequals  \card(A).$$

\subsection{Branched coverings and related mapping classes}
A mapping $f \colon \Omega \to \R^n$ is said to be
\begin{itemize}
\item[(i)] \emph{open} if it maps every open set in the domain to an open set in the target,
\item[(ii)] \emph{discrete} if the set of pre-images is a discrete set in the domain for every point in the target,
\item[(iii)] \emph{proper} if the pre-image of every compact set contained in $f(\Omega)$  is compact set in the domain, and
\item[(iv)] \emph{a branched covering map}, or more informally \emph{a branched covering}, if it is continuous, discrete, and open.
\end{itemize}
Note that by definition a branched covering $f \colon \Omega \to \R^n$ is locally injective outside its branch set
\begin{align*}
  B_f
  \colonequals \{ x \in \Omega \mid f \text{ is not a local homeomorphism at } x \}.
\end{align*}
The most elementary example of a noninjective proper branched cover is the \emph{winding mapping} $w_m \colon \R^n \to \R^n$ defined in cylindrical coordinates by the formula
\begin{align*}
  w_m(r, \theta, x_3, \ldots, x_n)
  = \bigl( r , m\theta, x_3, \ldots, x_n \bigr),
\end{align*}
with some given integer $m$ such that $\abs{m} \ge 2$. The study of continuous, open and discrete mappings has a solid history which can be studied more for instance from
\cite{ChurchHemmingsen1,BernsteinEdmonds78,Piergallini,AaltonenPankka} and the references therein. 

An important subclass of branched coverings is the class of quasiregular mappings.
A mapping $f \colon \Omega \to \R^n$ is called \emph{$K$-quasiregular} with $K \geq 1$ if 
\begin{itemize}
\item[(i)] it belongs to Sobolev space $W^{1,n}_{\text{loc}}(\R^n, \Omega)$, and
\item[(ii)] it satisfies the \emph{distortion inequality}
  $$\norm{Df(x)}^n \le K J_f(x)$$
  for almost every $x \in \Omega$.   
\end{itemize}
Above $\norm{Df(x)}$ refers to the operator norm of the weak differential $n \times n$ matrix $Df(x)$ at a point $x \in \Omega$, and $J_f(x) \colonequals \det Df(x)$ stands for the Jacobian determinant of $Df(x)$. 

For the basic knowledge on quasiregular mappings, we refer to \cite{Vuorinen,Rickman-book}. By the Reshetnyak theorem quasiregular mappings are branched coverings (\cite{Reshetnyak67} or \cite[Section IV.5, p.\ 145]{Rickman-book})
and so branched coverings can be seen as generalizations of quasiregular mappings, see
e.g.\ \cite{LuistoPankka-Stoilow}. For further discussion on quasiregular mappings and other related mapping classes we refer to
\cite{AstalaIwaniecMartinBook, HenclKoskela, IwaniecMartinBook, MRSYBook, VaisalaBook}.

The term \emph{branched cover(ing)} is widely used in the theory of quasiregular mappings to mean continuous, open and discrete
mappings; this terminology seems to originate from \cite{Rickman-book}.
However, the term is not standard even in closely related fields
and thus we will explore the nomenclature a bit. In particular, we wish to comment how a branched covering relates to \emph{covering maps}.

For \emph{proper} branched coverings the connection to covering maps is quite imminent. Indeed,
when a branched covering map $f \colon X \to Y$ is 
assumed to be proper then it is actually a covering map when restricted to the set
\begin{align*}
  X \setminus f^{-1}(f(B_f)).
\end{align*}
In particular, when $B_f = \emptyset$ a proper branched covering is always a covering map; for the
proof see Lemma \ref{lemma:NoBranchProperIsCover} below.
Note, however, that in general the restriction of a branched covering $f \colon X \to Y$
to the complement of $f \inv (f (B_f))$ does not yield a covering map; see e.g.\ \cite{Aaltonen} for some further discussion.

From this point of view we note that
in \cite{BonkMeyer} branched coverings are defined to be locally equivalent to winding maps,
in \cite{BernsteinEdmonds78} the mappings are studied only between closed manifolds which implies properness,
in \cite{Fox} branched covers are maps that are 'completions' of cover maps defined on an open dense subset of the domain,
and in \cite{Piergallini} branched covers are only studied in the PL-category where properness
also follows. This list should not be considered to be in any way exhaustive, but does demonstrate that
the term branched covering needs to be used carefully.
In our setting a branched covering needs not to be proper, but we do note that
every point in the domain always has a neighborhood basis of normal domains $U$ with the property
that the restriction of $f$ to $U$ is proper; see Section \ref{sec:NormalDomains} for details.
Thus, in a sense, a continuous, open and discrete mapping is locally similar to a covering map.

\subsection{Normal domains and path-lifting}
\label{sec:NormalDomains}

We follow the conventions of \cite{Rickman-book} and
say that a domain $U \subset X$ is a \emph{normal domain} for a branched covering $f \colon X \to Y$ if $U$ is compactly contained in $X$
and 
\begin{align*}
  \partial f (U) = f (\partial U).
\end{align*}
A normal domain $U$ is \emph{a normal neighborhood} of $x \in U$ if
\begin{align*}
  \overline{U} \cap f \inv (f(x)) = \{ x \}.
\end{align*}
If $Y$ is a metric space, then we denote by  $U(x,f,r)$ the component of the open set $f \inv (B_Y(f(x),r))$ containing $x$.
The existence of arbitrarily small normal neighborhoods 
is essential for the theory of branched covers. The following lemma
guarantees the existence of normal domains, the proof can be found in \cite[Lemma I.4.9, p.19]{Rickman-book} (see also \cite[Lemma 5.1.]{Vaisala}).
\begin{lemma}\label{lemma:TopologicalNormalDomainLemma}
  Let $X$ and $Y$ be locally compact complete path-metric
  spaces and $f \colon X \to Y$ a branched cover.
  Then for every point $x \in X$ 
  there exists a radius $r_0 > 0$ such that
  $U(x,f,r)$ is a normal neighborhood of $x$ for
  any $r \in (0,r_0)$.
  Furthermore,
  \begin{align*}
    \lim_{r\to 0}\diam U(x,f,r) = 0.
  \end{align*}
\end{lemma}

We prove several auxiliary results for a proper map from a domain $\Omega$ to $\R^n$
that are usually proven to hold when
a branched covering is restricted to a normal domain.
The proofs are essentially identical, since often the restriction to a normal
domain is needed only to guarantee that the map is proper. Our first lemma is a classical equidistribution result. For the
terminology and properties of the topological index we refer to
\cite[Section 4, p.16]{Rickman-book}. This lemma is a slight natural extension of
\cite[Proposition 4.10.(1), p.19]{Rickman-book} and \cite[Lemma 3.3]{Vuorinen2}.
\begin{lemma}\label{lemma:Equidistribution}
  Let $\Omega \subset \R^n$ be a domain and suppose
  $f \colon \Omega \to f(\Omega) \subset \R^n$ is a proper branched cover.
  Then 
  $$N(f) \colonequals \sup_{w \in f(\Omega)} \# (\Omega \cap f \inv(w)) < \infty \, ,$$
  and for any two points $y,z \in f(\Omega) \setminus f(B_f)$,
  \begin{align*}
    \# (\Omega \cap f \inv(y))
    = \# (\Omega \cap f \inv(z)).
  \end{align*}
\end{lemma}
\begin{proof}
  Fix the points $y, z \in f(\Omega) \setminus f(B_f)$. Then for the local index of an arbitrary point $x \in f^{-1}(y) \cup f^{-1}(z)$ we have
  $$\abs{i(x,f)} = 1.$$ 
  Thus, by \cite[Lemma 3.3]{Vuorinen2} we get $N(f) < \infty$, and
  \begin{align*}
    \# (\Omega \cap f \inv(y)) &= \sum_{x \in f^{-1}(y)} \abs{i(x,f)} = N(f) \\
                               &= \sum_{x \in f^{-1}(z)} \abs{i(x,f)} = \# (\Omega \cap f \inv(z)),
  \end{align*}	
  and the claim follows.
\end{proof}

Next we show that a proper branched covering with an empty branch is a covering map.
Se e.g.\ \cite{Ho} for further results related to the topic.
\begin{lemma}\label{lemma:NoBranchProperIsCover}
  Let $\Omega \subset \R^n$ be a domain and suppose
  $f \colon \Omega \to f(\Omega) \subset \R^n$ is a proper branched cover
  such that $B_f = \emptyset$. Then $f$ is a covering map.
\end{lemma}
\begin{proof}
  Fix a point $y \in f(\Omega)$. Then by Lemma~\ref{lemma:Equidistribution} 	
  $$N(f) = \# (\Omega \cap f \inv(y)) < \infty.$$
  By denoting $m \colonequals N(f)$ we may write
  $$f^{-1}(y) = \{ x_1,x_2, \ldots, x_m\}$$
  for some distinct points $x_1,x_2, \ldots, x_m$ in $\Omega$. Next, consider $x_i$-centric open balls
  $$B_i \colonequals B(x_i, r_i) \subset \Omega, \quad i=1,2, \ldots, m \, ,$$
  with positive radii and such that the restriction
  $$f_i = f|_{B_i} \colon B_i \to f(B_i)$$
  is a homeomorphism for every $i=1, \ldots, m$. Such a ball exists for every $x_i$ as $B_f = \emptyset$. Set
  $$W_y = \bigcap_{i=1}^{m} f(B_i).$$
  Then by the openness of $f$ and definition of the balls $B_i$ it follows that $W_y$ is open and it contains the point $y$. Take a ball
  $$V_y \colonequals B(y,\delta) \Subset W_y,$$
  and for each $i=1, \ldots, m$ set $U_i \colonequals f^{-1}(B_y)$.

  Since $\overline{V_y}$ is compactly contained in $f(\Omega)$ and $f$ is proper, each pre-image component of $f \inv (V_y)$ maps
  surjectively onto $V_y$.
  Thus $f^{-1}(V_y)$ is a disjoint union of the open sets $U_i$, and the restrictions
  $$f|_{U_i} \colon U_i \to V_y, \quad i=1, \ldots, m,$$
  are homeomorphisms. Thus, $f$ is a covering map.
\end{proof}

Finally, a fundamental technique in the study of branched covers is the path-lifting.
For the terminology and basic theory of this technique, we refer to
\cite[Section 3, p.32]{Rickman-book}\footnote{We remark that for their path-lifting methods,
  \cite{Rickman-book} assumes that the map is sense-preserving. This is merely a notational
  convention as a continuous, open and discrete mapping is always either sense-preserving or
  sense-reversing, see \cite[p. 18]{Rickman-book}.}.
\begin{lemma}\label{lemma:PathLifting}
  Let $\Omega \subset \R^n$ be a domain and suppose
  $f \colon \Omega \to f(\Omega) \subset \R^n$ is a proper branched cover.
  Then for any $\beta \colon [0,1] \to f(\Omega)$ and any
  $x \in \Omega \cap f \inv ( \beta(0) )$ there
  exists a path $\alpha \colon [0,1] \to \Omega$ for which
  $f \circ \alpha = \beta$ and $\alpha(0) = x$.
  Such a path is called a lift of $\beta$ (under $f$).
\end{lemma}
\begin{proof}
  By \cite[Corollary 3.3., p.34]{Rickman-book} 
  there exists a maximal lift $\gamma \colon I \to \Omega$ of
  $\beta$ such that $\gamma(0) = x$, where $I$ is a subinterval of $[0,1]$ containing $0$. We need to show
  that $I = [0,1]$. Towards contradiction suppose not, and assume first that $I$ is a closed interval
  $I = [0,a]$. But now since $a < 1$ we may take a maximal lift
  of $\beta|_{[a,1]}$ starting from the point $\gamma(a)$ and concatenate this lift to $\gamma$. This contradicts
  the maximality of the lift $\gamma$, and we deduce that $I$ must be open, i.e.\ of the form $[0,b)$.
  But now we note that by continuity of $f$, $|\gamma| \subset f \inv (| \beta |)$, and by the properness of
  $f$ we have that $f \inv (| \beta |)$ is a compact set. Thus we may conclude that the limit
  $\lim_{t \to b} \gamma(t)$ exists and is contained in $f\inv(\beta(b))$, and so we may extend the lift
  $\gamma$ to the closed interval $[0,b]$. This is again a contradiction with the maximality of $\gamma$, and
  so the original claim holds true; $I = [0,1]$ and we may choose $\alpha = \gamma$.
\end{proof}

\section{Main Theorems}
\label{sec:Main}

In this section we prove our main results, Theorem \ref{thm:VuorinenDim3} and Theorem \ref{thm:VuorinenNoBranch}.

\subsection{The results in dimension three}

We begin with a lemma that gives rise to a useful collection of large normal domains in the setting of the Vuorinen question.

\begin{lemma}\label{lemma:VuorinenDim3Lemma}
  Let $f \colon \Omega \to f(\Omega) \subset \R^n$ be a proper branched covering with $\Omega \subset \R^n$
  a domain. Let $K \subset f (\Omega)$ be a non-empty compact set and denote $C \colonequals f \inv (K) \subset \Omega$.
  Suppose $V \subset \Omega$ is a domain such that $C \subset V$,
  and denote $U \colonequals f \inv (f (V))$. Then
  \begin{enumerate}[(a)]
  \item if $\overline{V} \subset \Omega$ then $U$ is a normal domain for $f,$\label{part:NormalDomain}
  \item $U$ is connected and $f(U) = f(V)$,\label{part:Connected}
  \item $f|_{U} \colon U \to f(U)$ is a proper branched cover, and\label{part:ProperMap1}
  \item $f|_{\Omega \setminus \overline{U}} \colon \Omega \setminus \overline{U} \to f(\Omega) \setminus f(\overline{U})$
    is a proper branched cover
    with $\Omega \setminus \overline{U}$ connected.
    Moreover we have $ \Omega \setminus \overline{U}
    = f \inv (f (\Omega \setminus \overline{U})).$\label{part:ProperMap2}
  \end{enumerate}
\end{lemma}
\begin{proof}
  Note that since $V$ is a domain and $f$ an open map, $f(V)$ is open and so is its pre-image $U$ under the continuous map $f$.

  To prove part \eqref{part:NormalDomain} we need to show first that
  $\partial f(U) = f \partial (U)$ with $U$ connected.
  We note on the first hand that since $f$ is open, we have for any domain $A \subset \Omega$ that $\partial f(A) \subset f (\partial A)$,
  and so $\partial f(U) \subset f (\partial U)$. On the other hand, for any point $y \in f (\partial U)$ we note that
  if $U \cap f \inv ( y ) \neq \emptyset$, then $U$ is a neighborhood of one of the pre-images of $y$ and so by openness of
  $f$, $f(U)$ is a neighborhood of $y$ in $f(U)$ since $f(U) \subset V$. This implies that $U = f \inv (f(V))$ is a neighborhood of
  all of the points in the pre-image of $y$, and thus $\partial U \cap f \inv ( y ) = \emptyset$. This is a contradiction
  and so we must have that $U \cap f \inv ( y ) = \emptyset$, and so $y \in \partial f (U)$ since $y \in \overline{f(U)}$.
  
  Next, to see that $U$ is connected, we note that $f(V)$ is a connected domain containing $f(K)$ and so
  for any point $x \in U$ we may connect $f(x)$ and $f(K)$ with a path $\alpha \colon [0,1] \to f(V)$. Now by Lemma
  \ref{lemma:PathLifting} the path $\alpha$ has a lift $\tilde\alpha\colon [0,1] \to \Omega$ with $\tilde \alpha(0) = x$
  and by the definition of $U$,
  $|\tilde \alpha| \subset U$. On the other hand $\tilde \alpha(1) \in C \subset V$, and so each point $x \in U$ can be connected
  with a path to an interior point of the connected set $V$. This implies that $U$ is also connected, and now since $V \subset f \inv (f (V))$,
  we also see that $f(U) = f(V)$. This concludes the proof of parts \eqref{part:NormalDomain} and \eqref{part:Connected}.

  For part \eqref{part:ProperMap1} we first note that the restriction of a branched covering to a domain is a branched covering. To show the
  properness of the restriction $f|_{U} \colon U \to f(U)$ we fix a compact set $A \subset f(U)$ and note that $f \inv (A) \subset \Omega$
  is compact since $f$ is proper. Now as $U = f \inv (f(U))$, we have that $f \inv (A) \subset U$, and so we see that $(f|_U) \inv A$ is compact.
  Thus $f|_U$ is a proper map.

  Finally for part \eqref{part:ProperMap2} we again note that the restriction of a branched covering to a domain is a branched covering.
  Since by part \eqref{part:NormalDomain} we have $\partial f(U) = f (\partial U)$, we see that also
  \begin{align}\label{eq:BoundaryForAnnularDomain}
    f (\partial (\Omega \setminus \overline{U} ))
    = \partial f (\Omega \setminus \overline{U}),
  \end{align}
  where the boundary is taken relative to the domain $f(\Omega)$. As in part \eqref{part:ProperMap1}, the properness will follow
  after we show that
  \begin{align*}
    \Omega \setminus \overline{U}
    = f \inv (f (\Omega \setminus \overline{U})).
  \end{align*}
  The inclusion
  \begin{align*}
    \Omega \setminus \overline{U}
    \subset f \inv (f (\Omega \setminus \overline{U}))
  \end{align*}
  is trivial, so fix a point $x \in f \inv (f (\Omega \setminus \overline{U}))$.
  Suppose, towards contradiction, that $x \notin \Omega \setminus \overline{U}$. Then necessarily $x \in \overline{U}$,
  whence either $x \in \partial U$ or $x \in U$. In the first case, we would have
  by applying \eqref{eq:BoundaryForAnnularDomain} twice that $x \in \partial (\Omega \setminus \overline{U})$,
  which is not possible as we chose $x$ from within the domain $\Omega \setminus \overline{U}$.
  In the second case, if $x \in U$, then by the definition of $U$, $f \inv ( f(x) ) \subset U$, and
  so $x \notin \Omega \setminus \overline{U}$, which again goes against our assumptions. Thus we conclude
  that $x \in \Omega \setminus \overline{U}$ and so
  \begin{align*}
    \Omega \setminus \overline{U}
    \supset f \inv (f (\Omega \setminus \overline{U})).
  \end{align*}
  This concludes the proof of the claim.
\end{proof}

Our proof relies on deep results in low-dimensional topology, namely
Proposition \ref{prop:3dimNoTorsion}. For the statement of the result
we need some auxiliary concepts.
We refer to \cite{Hatcher} for basic definitions and theory of homotopy and
denote the homotopy groups of a space $X$ by $\pi_k(X)$ for $k \in \N$.
\begin{definition}\label{def:TorsionFreeAtInfinity}
  We say that a domain $\Omega \subset \R^n$ has \emph{torsion free fundamental group at infinity} if
  for any compact set $K \subset \Omega$ there exists a domain $V \supset K$ with $\overline{V} \subset \Omega$
  and such that $\pi_1(\Omega \setminus \overline{V})$ is torsion free; recall that
  a group is torsion free if no non-zero element $g$ has the property that $g^j = e$ for some $j \in \N$.
\end{definition}
The nomenclature of this definition is motivated by a similar definition of a space being \emph{simply connected at infinity};
see e.g.\ \cite{Edwards}.

The following proposition is the fundamental observation in the proof of our first main theorem.
We wish to emphasize that Proposition \ref{prop:NoTorsionImpliesVuorinen} is valid in all dimensions.
\begin{proposition}\label{prop:NoTorsionImpliesVuorinen}
  Let $f \colon B^n \to f(B^n) \subset \R^n$ be a proper branched covering with $n \ge 2$. Suppose that $f(B^n)$ has
  torsion free fundamental group at infinity. If $B_f$ is compact, then $f$ is a homeomorphism.
\end{proposition}
\begin{proof}
  Since $B_f$ is compact and $f$ is a continuous proper map, both $f(B_f)$ and $f\inv (f(B_f))$ are also compact.
  Thus there exists $r_0 > 0$ such that for all $r \in [r_0,1)$, $f \inv (f (B_f)) \subset B_r$.
  For such $r \in [r_0,1)$ we denote
  \begin{align*}
    U_r \colonequals f \inv (f (B_r)),
    \quad
    \text{ and }
    \quad
    E_r \colonequals B \setminus \overline{U}_r.
  \end{align*}
  By Lemma \ref{lemma:VuorinenDim3Lemma} \eqref{part:ProperMap1} and \eqref{part:ProperMap2}
  the restriction of $f$ to either one of these  domains will be a proper branched cover.

  Since we assumed $f(B)$ to be torsion free at infinity, there exists a compact set $K \subset f(B)$ such that
  $K \supset f(\overline{B}_{r_0})$ and $f(B) \setminus K$ has a torsion-free fundamental group.
  We fix now a radius $s \in (r_0,1)$ for which $K \subset f(B)$, and take $R \in (s,1)$ to be such that
  $\overline{U}_s \subset B_R$; see Figure \ref{fig:TorsionConstruction}.
  Since $f \colon B \to \R^n$ is a proper branched cover, we note that all points in
  $f(B) \setminus f(B_f)$ have an equal amount of pre-images by Lemma \ref{lemma:Equidistribution};
  call this number $k$.
  In particular we note that since $B_f \subset B_s \subset B_R$, all the points in
  \begin{align*}
    f(B) \setminus f(\overline{B}_R)
    = f(E_R)
  \end{align*}
  have an equal amount of pre-images in $B$. But now, since by Lemma \ref{lemma:VuorinenDim3Lemma} \eqref{part:ProperMap2}
  $f \inv (f(E_R)) = E_R$, we note that all the points in $f(E_R)$ have $k$ pre-images in $E_R$. Thus we see that
  $f|_{E_R} \colon E_R \to f(E_R)$ is a $k$-to-one covering map.

  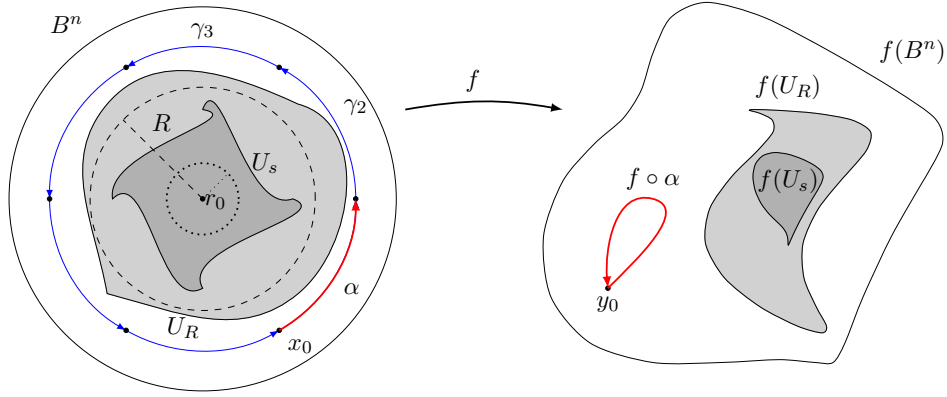
\begin{figure}[h!]
    \centering
    \resizebox{\textwidth}{!}{
      \begin{tikzpicture}
        \def\dotsize{0.05} 
        
        \begin{scope}[scale=0.7]
          
          \begin{scope}[shift={(-6,0)}]
            \def\rn{0.8}  
            \def\r{2.0}   
            \def\s{2.5}   
            \def\R{2.5}   
            \def\pr{3.4}  
            \def\one{4.3} 

            \begin{scope}[]
              \coordinate(r1_N) at ( 0 , \r);          
              \coordinate(r1_E) at ( \r, 0 );          
              \coordinate(r1_S) at ( 0 ,-\r);          
              \coordinate(r1_W) at (-\r, 0 );
            \end{scope}

            \begin{scope}[rotate=45,scale=1.2]
              \coordinate(R1_N) at ( 0 , \R);
              \coordinate(R1_E) at ( \R, 0 );
              \coordinate(R1_S) at ( 0 ,-\R);
              \coordinate(R1_W) at (-\R, 0 );
            \end{scope}

            \draw[fill, black!60, opacity=0.3] 
            (R1_N) to [out = 50, in=155] 
            (R1_E) to [out = -15, in=40] 
            (R1_S) to [out = 220, in=-15] 
            (R1_W) to [out = 100, in=230] 
            (R1_N);
            \draw[]
            (R1_N) to [out = 50, in=155] 
            (R1_E) to [out = -15, in=40] 
            (R1_S) to [out = 220, in=-15] 
            (R1_W) to [out = 100, in=230] 
            (R1_N);
            \node[below] at (260:1*\R) {$U_R$} ;

            \draw[fill, black!60, opacity=0.3] 
            (r1_N) to [out = 70, in=200] 
            (r1_E) to [out = -20, in=120] 
            (r1_S) to [out = 210, in=45] 
            (r1_W) to [out = 100, in=-10] 
            (r1_N);
            \draw[] 
            (r1_N) to [out = 70, in=200] 
            (r1_E) to [out = -20, in=120] 
            (r1_S) to [out = 210, in=45] 
            (r1_W) to [out = 100, in=-10] 
            (r1_N);
            \node[] at (30:0.8*\r) {$U_s$} ;

            \foreach \i in {1,...,6} {
              \draw[blue, -{latex}] (\i*60:\pr) arc (\i*60:(\i+1)*60:\pr);
              \draw[fill] (60*\i:\pr) circle [radius=\dotsize];
            }

            \node[below right] at (60*5:\pr) {$x_0$};
            \draw[red,thick,-{latex}] (60*5:\pr) arc (5*60:(5+1)*60:\pr);
            \node[below right] at (60*5.5:\pr) {$\alpha$};

            \node[above right] at (60*6.5:\pr) {$\gamma_2$};
            \node[above] at (60*7.5:\pr) {$\gamma_3$};
            
            \draw[dotted,thick] (0,0) circle [radius=\rn];
            \node[below ] at (45:0.5*\rn) {$r_0$} ;
            \draw[dotted] (0,0) -- (45:\rn);          

            \draw[dashed] (0,0) circle [radius=\R];
            \draw[dashed] (0,0) -- (135:\R);          
            \node[above right] at (135:0.75*\R) {$R$} ;          

            \draw[] (0,0) circle [radius=\one];
            \draw[fill] (0,0) circle [radius=\dotsize];
            \node[above left] at (125:\one) {$B^n$};
          \end{scope}

          \draw[-{latex}, thick] (-1.5,2) to [out=10,in=170] (2,2);
          \node[above] at (0,2.2) {$f$};

          \begin{scope}[shift={(6,0)}]
            \def\outer{4.0} 
            \def\fBR{2.5}   
            \def\fBs{1.0}   
            
            \draw[domain = 0:360, smooth,fill, black!60, opacity=0.3]  plot
            ( {0.5*\fBR*cos(\x) + 0.3*\fBR*sin(3*\x) + 1 },
            { \fBR*sin(\x) + 0.2*\fBR*cos(2*\x) } );
            \draw[domain = 0:360, smooth] plot
            ( { 0.5*\fBR*cos(\x) + 0.3*\fBR*sin(3*\x) + 1 },
            { \fBR*sin(\x) + 0.2*\fBR*cos(2*\x) } );
            \node[] at (1,2.5) {$f(U_R)$};

            \draw[domain = 0:360, smooth,fill, black!60, opacity=0.3]  plot
            ( { 0.6*\fBs*cos(\x) + 0.3*\fBs*sin(2*\x) + 1 },
            { \fBs*sin(\x) + 0.1*\fBs*cos(3*\x) } );
            \draw[domain = 0:360, smooth] plot
            ( { 0.6*\fBs*cos(\x) + 0.3*\fBs*sin(2*\x) + 1 },
            { \fBs*sin(\x) + 0.1*\fBs*cos(3*\x) } );
            \node[] at (1,0.4) {$f(U_s)$};
            
            \draw[domain = 0:360, smooth] plot
            ( { \outer*cos(\x) + 0.2*\outer*sin(3*\x) },
            { \outer*sin(\x) + 0.1*\outer*cos(4*\x) } );
            \node[above right] at (45:\outer) {$f(B^n)$};
            
            \def\yx{-3} 
            \def\yy{-2} 
            \coordinate (y0) at (\yx,\yy);
            \coordinate (y1) at (\yx+1,\yy+2);
            \node[below] at (y0) {$y_0$};
            \draw[fill] (y0) circle [radius=\dotsize];
            \draw[thick, red, -{latex}]
            (y0) to [out = 45, in=-15]
            (y1) to [out = 165, in=95] 
            (y0);
            \node[above] at (y1) {$f \circ \alpha$};
            
          \end{scope}
        \end{scope}
      \end{tikzpicture}
    }
    \caption{Essential objects in the proof of
      \ref{prop:NoTorsionImpliesVuorinen}.}
    \label{fig:TorsionConstruction}
  \end{figure}

  Fix a point $x_0 \in E_R$ and denote $f \inv (f(x_0)) \equalscolon \{ x_0, x_1, \ldots, x_{k-1} \}$.
  As $E_R$ is path-connected by Lemma \ref{lemma:VuorinenDim3Lemma} \eqref{part:Connected}, we may now take an injective path $\alpha \colon [0,1] \to E_R$
  with $\alpha(0) = x_0$ and $\alpha(1) = x_1$. The image of this loop, $\beta \colonequals f \circ \alpha \colon [0,1] \to fE_R$,
  is a loop based at $y_0 \colonequals f(x_0)$. If $\beta$ was zero-homotopic in $f(E_R)$, it would also be zero-homotopic in the larger
  domain $f(E_s)$ and we could lift the homotopy with
  the covering map $f|_{E_s} \colon E_s \to f(E_s)$ into a homotopy in $E_s$ contracting the path $\alpha$ to a point without
  changing the endpoints of the path, see \cite[Proposition 1.30]{Hatcher}.
  This is not possible when $k \geq 2$, and so we must have $[\beta] \neq 0$ in $\pi_1(f(E_s), y_0)$.
  Likewise, since $f$ is a proper map, its restriction to $B \setminus f \inv (f (B_f))$ is also a cover map and so $\beta$
  is not zero-homotopic also in $f(B) \setminus K \subset f(B) \setminus f(B_f)$.

  Next we construct a loop $\gamma \colon [0,m] \to E_R$, see again Figure \ref{fig:TorsionConstruction}.
  We set first $\gamma_1 = \alpha$. Then, when $\gamma_k \colon [0,k] \to E_R$
  has been defined and if $\gamma_k(0) \neq \gamma_k(1)$, we define $\gamma_{k+1}$ by lifting the path $\beta$ from the point
  $\gamma_k(1)$ and concatenating that lift to $\gamma_k$. Since the covering map is a local homeomorphism, this procedure
  is well defined and since $f|_{E_R}$ is $k$-to-one, it terminates after at most $k$ steps.

  But now we note that
  \begin{align*}
    |\gamma|
    \subset E_R
    \subset B \setminus \overline{B}_R
    \subset E_r,
  \end{align*}
  and so $\gamma$ can be contracted to a point in the annulus
  $ B \setminus \overline{B}_R$ and thus in $E_r$.
  This contracting homotopy can then be pushed with the covering map $f|_{E_r}$ into $f(E_r)$,
  and so we see that
  \begin{align*}
    0
    = [f \circ \gamma]
    = [\beta]^m,
  \end{align*}
  and so we see that $[\beta]$ is a non-trivial torsion element in $\pi_1(E_r, y_0)$. Since,
  as noted before, $[\beta] \neq 0$ also in $ f(B) \setminus K$ and clearly
  $[\beta]^m = 0$ in $f(B) \setminus K$, we see that $[\beta]$ is also a non-trivial torsion element
  in $f(B) \setminus K$.
  This is a contradiction and so the original claim holds.
\end{proof}

Our proof in dimension three relies on the following result of Papakyriakopoulos, see \cite[Corollary 31.8]{Papakyriakopoulos}.
\begin{proposition}\label{prop:3dimNoTorsion}
  Let $\Omega \subset \R^3$ be a domain. Then $\pi_1(\Omega)$ is torsion free.
\end{proposition}
Note that by Proposition \ref{prop:3dimNoTorsion}, since any domain in $\R^3$ has torsion
free fundamental group, they in particular have torsion free fundamental group at infinity.
This yields the proof of our first main theorem.
\begin{proof}[Proof of Theorem \ref{thm:VuorinenDim3}]
  Let $f \colon B^3 \to f(B^3) \subset \R^3$ be a proper branched covering and denote $Y \colonequals f(B^3)$.
  By Lemma \ref{prop:3dimNoTorsion} we know that for any compact set $K \subset Y$, the
  fundamental group of $Y \setminus K$ is torsion-free. Thus $Y$ has torsion-free fundamental group
  at infinity, and the claim follows from Proposition \ref{prop:NoTorsionImpliesVuorinen}.
\end{proof}

\begin{remark}
  Note that if the target of $f$ is not assumed to be a manifold,
  we may take the the universal covering map $p \colon \bS^2 \to P^2$ onto the
  projective plane $P^2$ and let $f$ be the \emph{cone map} (see e.g.\ \cite{LuistoPrywes} for the terminology)
  \begin{align*}
    \cone(p) \colon \cone(\bS^2)
    = \overline{B^3} \to \cone(P^2).
  \end{align*}
  The mapping $f$ restricted
  to the open ball $B^3$ is a proper branched covering onto a space which is an open 3-manifold
  outside one singular point. Furthermore $B_f = \{ 0 \}$, so in particular the branch is
  non-empty but compact. Similar examples appear from universal covers of homology spheres.
  Thus we must assume that the image of $f$ is a manifold. We do remark that we do not
  know if the Vuorinen question holds for mappings $f \colon B^3 \to N$ where $N$ is
  a $3$-manifold not necessarily embeddable into $\R^3$.
\end{remark}

\subsection{Proof of the case $B_f = \emptyset$}
\label{sec:BranchlessVuorinen}

The proof in the setting of no branch set relies on the
fact that the mapping $f \colon B^n \to f(B^n) \subset \R^n$ is in fact a covering
map defined on a contractible $n$-manifold. This observation can be used to show that the target of the mapping is actually \emph{a Eilenberg-MacLane space $K(G,1)$}, i.e., a path connected space whose fundamental group is isomorphic to a group $G$ and which has contractible universal covering space, see for instance \cite[p.87 onwards]{Hatcher}.
After this we use the notion of Eilenberg-MacLane spaces to rule out
the examples with torsion-unfree fundamental groups which could
potentially arise in higher dimensions. The following statement can be found in \cite[Proposition 2.45, p.149]{Hatcher}.

\begin{proposition}\label{prop:CW-noTorsion}
  Let $Y$ be a finite-dimensional CW-complex. If $Y$ is a $K(G,1)$-space,
  then $G = \pi_1(Y)$ is torsion-free.
\end{proposition}

The advantage of Proposition~\ref{prop:CW-noTorsion} is that it can be used to provide the torsion-freeness of the fundamental group of the target without any additional assumption on the dimension. Whenever this property of the fundamental group of the target is verified we can give a positive answer to Vuorinen question with the techniques introduced in this paper. Besides Proposition~\ref{prop:CW-noTorsion} we need also the following simple lemma to prove Theorem~\ref{thm:VuorinenNoBranch}. We note that Lemma \ref{lemma:YhasFiniteFG} is known to the experts in the field,
but we have not seen it explicitly stated in the literature so we provide a proof for the convenience of the reader.
\begin{lemma}\label{lemma:YhasFiniteFG}
  Let $f \colon B^n \to f(B^n) \subset \R^n$ be a proper branched covering, $n \geq 2$ and denote $Y = f(B^n)$. Then $\pi_1(Y)$ is finite.	
\end{lemma}	
\begin{proof}
  Take the universal covering $\tilde Y$ of $Y$. Since $B^n$ is simply connected, $f$ lifts to a map $\tilde f \colon B^n \to \tilde Y$.
  We show first that the lift $\tilde f$ is surjective. Towards contradiction
  suppose it is not, whence there exists a point $z_0 \in \partial \tilde f (B^n)$,
  and in particular for all $r>0$, $\tilde f \inv (\overline{B}(z_0,r))$ is not a compact subset of $B^n$.
  Fix a covering neighborhood $U$ of $p(z_0)$, and let $V$ be the $z_0$-component of $p \inv ( U )$. By the definition of a covering neighborhood,
  the restriction $p|_V \colon V \to U$ is a homeomorphism. Now for a radius $r>0$ such that $\overline{B}(z_0,r) \subset V$ we note that
  $p(\overline{B}(z_0,r))$ is a compact subset of $Y$ such that $f \inv (p \overline{B}(z_0,r))$ is not a compact subset of $B^n$ since
  $f = p \circ \tilde f$
  and $\tilde f \inv \overline{B}(z_0,r)$ is not compact. This is a contradiction with the properness of $f$ and so $\tilde f$ must be surjective.

  Next we note that since $\tilde Y$ is the universal cover of $Y$,
  $\# p\inv(y) = \# \pi_1(Y)$ for any $y \in Y$, see e.g.\ \cite[Proposition 1.32]{Hatcher}.
  Since $f = p \circ \tilde f$ and $f$ is surjective, we conclude that $N(f) \geq \# \pi_1(Y)$,
  which implies with Lemma \ref{lemma:Equidistribution}
  that $\pi_1(Y)$ is finite.
\end{proof}

\begin{proof}[Proof of Theorem \ref{thm:VuorinenNoBranch}] 	
  Suppose $f \colon B^n \to f(B^n) \subset\R^n$ is a proper branched covering such that
  $B_f = \emptyset$. Denote 
  $$Y \colonequals f(B^n) \subset \R^n \quad \text{and} \quad G \colonequals \pi_1(Y).$$ 
  Now, on the one hand, by Lemma~\ref{lemma:NoBranchProperIsCover} $f$ is a covering map, and since $B^n$ is simply connected it is the universal cover
  of $Y$. This means, by definition, that $Y$ is an $K(G,1)$-space; see \cite[p.87]{Hatcher}.
  
  On the other hand, as $Y \subset \R^n$ is a domain, its Whitney decomposition into dyadic cubes gives it the structure of a CW-complex. Thus by Proposition~\ref{prop:CW-noTorsion} the fundamental group $G$ of $Y$ has no torsion.

  Finally, by Lemma~\ref{lemma:YhasFiniteFG} we note that $\pi_1(Y)$ must be finite. In
  particular, we may deduce that as a finite torsion-free group $\pi_1(Y)$ is trivial and therefore $Y$ is simply connected.
  This shows that $f$ is a covering map from a path connected space to a simply connected space, which implies that
  it is actually a global homeomorphism, see e.g.\ \cite[Theorem 1.38]{Hatcher}. This proves the claim. 
\end{proof}

Lastly, we describe a classical example by Church and Timourian (\cite{ChurchTimourian}), based on deep
results of Cannon et.\ al. (see e.g.\ \cite{Cannon} and the references therein),
of a continuous, open and discrete mapping $\bS^5 \to \bS^5$ with very nontrivial
branch behavior. For further discussion and a precise definition on this map see e.g.\
\cite{AaltonenPankka} or \cite{LuistoPrywes}.
\begin{example}\label{ex:Nemesis}
  Let $p \colon \bS^3 \to P$ be the universal covering map of the \emph{Poincar\'e homology sphere} $P$.
  Due to the work of Cannon and Edwards (see e.g.\ \cite{Cannon}) the double suspension $f$ of this covering map $p$ is a branched cover
  between $5$-spheres and has a branch set equal to the suspension of the two branch points of the single suspension of the covering map $p$.
  Thus the branch set $B_{f}$ of $f$ is PL-equivalent to $\bS^1$ and so we see that $f$
  is a branched covering between two spheres with a branch set of codimension four.

  Note that the image of the branch set $B_{S^2(f)}$ is complicated even though it is a Jordan curve
  in $\bS^5$ since its complement has a fundamental group of 120 elements -- the binary icosahedral group.
  Furthermore, removing the branch set and its image from the domain and range, respectively, gives rise to a covering map
  $\R^5 \setminus \{ (x,0,0,0,0) \} \to \R^5$ whose range has a finite non-trivial fundamental group.
\end{example}


\def\cprime{$'$}\def\cprime{$'$}

\end{document}